\documentclass[11pts]{amsart}

\usepackage[utf8]{inputenc}
\usepackage[english]{babel}
\usepackage [autostyle, english = american]{csquotes}
\MakeOuterQuote{"}
\usepackage[a4paper,twoside,top=1.2in, bottom=1.1in, left=1.4in, right=1.4in]{geometry}

\usepackage{graphicx}
\usepackage{tikz}
\usetikzlibrary{graphs}
\usepackage{caption} % Better handling of empty figure captions.
\usepackage{amsmath}
\usepackage{amsthm}
\usepackage{amssymb}
\usepackage{esint} % For the average integral
\usepackage{mathrsfs} % Just for cool sigma-algebra
\usepackage[dvipsnames]{xcolor}
\usepackage{array}
\usepackage{hhline}
\usepackage[shortlabels]{enumitem}
\usepackage{bm}
\usepackage{comment} % For comment environment
\usepackage[toc,page]{appendix} % Appendices
\usepackage{xparse} % In order to use ExplSyntax (in older latex versions)
\usepackage{mathtools}
\usepackage{fancyhdr} % Header and footer of pages
\usepackage{ifthen} % Ifthen construct
\usepackage{forloop} % forloop construct
\usepackage{xstring}
\usepackage{emptypage} % Leave completely empty the empty pages
\usepackage{setspace}
\usepackage[initials,alphabetic]{amsrefs} % For alphanumeric labels: alphabetic
\usepackage[normalem]{ulem}
\usepackage{import}
\usepackage{soul}

\setstcolor{red}

\usepackage{tikz}
\usetikzlibrary{arrows,shapes,patterns,calc,fadings,decorations.pathreplacing,decorations.markings,decorations.pathmorphing,backgrounds}

\usepackage{xcolor}

\usepackage{hyperref} % References become hyperlinks.
\hypersetup{
	colorlinks = true,
	linkcolor = {blue},
	urlcolor = {red},
	citecolor = {blue}
}

\usepackage{thmtools}

\usepackage[nameinlink,capitalise]{cleveref} % References show what kind of thing is being referenced (theorem, lemma, proposition...)

\usepackage[textsize=tiny]{todonotes}

\newcounter{results}[section] % Uniform counters for lemmas, theorems, propositions etc

\theoremstyle{plain}
\newtheorem{theorem}[results]{Theorem}
\newtheorem{lemma}[results]{Lemma}
\newtheorem{proposition}[results]{Proposition}
\newtheorem{corollary}[results]{Corollary}

\newtheorem*{theorem*}{Theorem}
\newtheorem*{lemma*}{Lemma}
\newtheorem*{proposition*}{Proposition}
\newtheorem*{corollary*}{Corollary}
\newtheorem*{exercise*}{Exercise}
\newtheorem*{fact*}{Fact}
\newtheorem*{claim*}{Claim}
\newtheorem*{observation*}{Observation}
\newtheorem*{conjecture*}{Conjecture}

\theoremstyle{remark}
\newtheorem{remark}[results]{Remark}

\newtheorem*{remark*}{Remark}
\newtheorem*{question*}{Question}

\theoremstyle{definition}
\newtheorem{definition}[results]{Definition}

\newtheorem*{definition*}{Definition}
\newtheorem*{example*}{Example}

\numberwithin{equation}{section}

\newtheorem{innercustomthm}{Theorem}
\newenvironment{customthm}[1]
  {\renewcommand\theinnercustomthm{#1}\innercustomthm}
  {\endinnercustomthm}

\ifdefined\comma 
        \renewcommand{\comma}{\ensuremath{\, \text{, }}}
\else 
        \newcommand{\comma}{\ensuremath{\, \text{, }}}
\fi

\newcommand{\N}{\ensuremath{\mathbb N}}%Natural numbers
\newcommand{\Z}{\ensuremath{\mathbb Z}}%Integers
\newcommand{\R}{\ensuremath{\mathbb R}}%Real numbers
\newcommand{\id}{\ensuremath{\mathrm{id}}}% Identity
\DeclareMathOperator{\Out}{Out}

\DeclareMathOperator{\MCG}{MCG}

\colorlet{myGray}{gray}
\colorlet{myBlue}{blue}
\colorlet{myBlack}{black}
\colorlet{myBackground}{gray!10}

\DeclareMathOperator{\Cay}{Cay}
\DeclareMathOperator{\Inn}{Inn}
\DeclareMathOperator{\Wh}{Wh}
\DeclareMathOperator{\sep}{sep}
\DeclareMathOperator{\prim}{prim}
\DeclareMathOperator{\cut}{cut}
\DeclareMathOperator{\Aut}{Aut}
\DeclareMathOperator{\Axis}{Axis}

\newcommand{\ES}[0]{\mathcal{ES}}

\DeclareMathOperator{\Core}{Min}

\newcommand{\cC}[0]{\mathcal{C}}
\newcommand{\bZ}[0]{\mathbb{Z}}

\renewcommand{\Omega}[0]{\Wh}

\newcommand{\entrye}{e_\mathrm{entry}}
\newcommand{\exite}{e_\mathrm{exit}}
\newcommand{\entryv}{v_\mathrm{entry}}
\newcommand{\exitv}{v_\mathrm{exit}}

\title{The infinite dimensional geometry of conjugation invariant generating sets}
\author{Sabine Chu}
\author{George Domat}
\author{Christine Gao}
\author{Ananya Prasanna}
\author{Alex Wright}
    
\begin{document}

\begin{abstract}
    We consider a number of examples of groups together with an infinite conjugation invariant generating set, including: the free group with the generating set of all separable  elements; surface groups with the generating set of all  non-filling curves; mapping class groups and outer automorphism groups of free groups with the generating sets of all reducible elements; and groups with suitable actions on Gromov hyperbolic spaces with a generating set of elliptic elements. Building on work of Brandenbursky-Gal-K\c{e}dra-Marcinkowski, in these Cayley graphs we show that there are quasi-isometrically embedded copies of $\Z^m$ for all $m\geq1$. A corollary is that these Cayley graphs have infinite asymptotic dimension.
    
    By additionally building a new subsurface projection analogue for the free splitting graph, which is valued in the above Cayley graph of the free group and may be of independent interest, we are able to recover  Sabalka-Savchuk's result that the edge-splitting graph of the free group has quasi-isometrically embedded copies of $\Z^m$ for all $m\geq1$.
%
    %Our analysis for Cayley graphs builds on a construction of Brandenbursky-Gal-K\c{e}dra-Marcinkowski using quasi-morphisms. 
    %We observe in particular that the Cayley graph of  a closed surface group with the generating set of all simple closed curves is not hyperbolic, answering a question of Margalit-Putman.
\end{abstract}

\maketitle

\thispagestyle{empty}

%\setcounter{tocdepth}{1}
%\tableofcontents

%%%%%%%%%%%%%%%%%%%%%%%%%%%%%%%%%%%%%%%%%%%%%%%%%%%%%%%%%%%%%%%%
%%%%%%%%%%%%%%%%%%%%%%%%%%%%%%%%%%%%%%%%%%%%%%%%%%%%%%%%%%%%%%%%
\section{Introduction}\label{sec:intro}
%%%%%%%%%%%%%%%%%%%%%%%%%%%%%%%%%%%%%%%%%%%%%%%%%%%%%%%%%%%%%%%%
%%%%%%%%%%%%%%%%%%%%%%%%%%%%%%%%%%%%%%%%%%%%%%%%%%%%%%%%%%%%%%%%

The goal of this paper is to show that certain spaces arising in the context of Gromov hyperbolicity are in fact so far from being Gromov hyperbolic that they have quasi-flats of all dimensions. These flats are built using  Abelian behavior and certified using quasi-morphisms. 

%%%%%%%%%%%%%%%%%%%%%%%%%%%%%%%%%%%%%%%%%%%%%%%%%%%%%%%%%%%%%%%%
\subsection{Abelian behavior for conjugation invariant generating sets}
%%%%%%%%%%%%%%%%%%%%%%%%%%%%%%%%%%%%%%%%%%%%%%%%%%%%%%%%%%%%%%%%
%%%%%%%%%%%%%%%%%%%%%%%%%%%%%%%%%%%%%%%%%%%%%%%%%%%%%%%%%%%%%%%%
Let $G$ be any group and $S$ any conjugation invariant generating set for $G$. This $S$ is typically infinite. 

By definition, the Cayley graph $\Cay(G,S)$ has edges from $g$ to $sg$ for $g\in G$ and $s\in S$. Suppose $g,h\in G$. If $G$ were Abelian, we would have an edge from $gh$ to $gsh$, because $gsh=sgh$. The assumption that $S$ is conjugation invariant also gives an edge from $gh$ to $gsh$, now because 
$$gsh = (gsg^{-1})(gh).$$
Thus, Cayley graphs for conjugation invariant generating sets have some commonalities with Cayley graphs of Abelian groups. In particular, we get that if $p_1, \ldots, p_m \in G$ then the map
$$f:\bZ^m \to \Cay(G,S)$$
    defined by 
\begin{equation}\label{E:fdef}
f(i_1, \ldots, i_m)= p_1^{i_1} \cdots p_m^{i_m}
\end{equation}
is Lipschitz.

%%%%%%%%%%%%%%%%%%%%%%%%%%%%%%%%%%%%%%%%%%%%%%%%%%%%%%%%%%%%%%%%
\subsection{From quasi-morphisms to quasi-flats}
%%%%%%%%%%%%%%%%%%%%%%%%%%%%%%%%%%%%%%%%%%%%%%%%%%%%%%%%%%%%%%%%

In \cite[Theorem 3.3]{CancelationNorm} Brandenbursky-Gal-K\c{e}dra-Marcinkowski consider a group $G$ together with a generating set consisting of finitely many conjugacy classes of $G$. They show that if the group admits $m$ linearly independent homogeneous quasi-morphisms, then there is a quasi-isometric embedding of $\bZ^m$ into this Cayley graph. The proof, which is short and elegant, also gives the following more general statement. 

\begin{theorem}\label{thm:qiembedding}
Let $S$ be a conjugation invariant generating set for a group $G$. Suppose $G$ has $m$ linearly independent homogeneous quasi-morphisms that are bounded on $S$. Then there is a quasi-isometric embedding of $\bZ^m$ into $\Cay(G,S)$. 
\end{theorem}

Note that if $S$ consists of finitely many conjugacy classes, the requirement that the homogeneous quasi-morphisms are bounded on $S$ is automatic; see for example \cite[Section 2.2]{scl} for this and other standard facts on quasi-morphisms.   

The outline of the proof is as follows. Using linear combinations as in \cite[Lemma 3.10]{Autonomous} we obtain quasi-morphisms $q_1, \ldots, q_m$ with the same properties such that there are group elements $p_1, \ldots, p_m$ with $$q_k(p_j)=\delta_{kj}.$$
With this choice of the $p_j$, consider the map $f$ defined in \eqref{E:fdef}, and observe that 
$$q_k(f(i_1, \ldots, i_m) f(i_1', \ldots, i_m')^{-1})= i_k-i_k'+ O(1).$$
Here the  $O(1)$ error term depends only on $m$ and the defect of $q_k$. 
Since $q_k$ is bounded on $S$ it defines a Lipschitz map on $\Cay(G,S)$, so this is enough to certify that $f$ is a quasi-isometric embedding. 

One example of a conjugation invariant generating set is the set of all commutators, considered as a generating set for the commutator subgroup, and this leads to the theory of commutator length. In this context (and the more general context of verbal subgroups) Calegari and Zhuang previously applied the Lipschitz property of \cref{E:fdef} to give connections between asymptotic cones and quasi-morphisms \cite[Section 3]{CalegariZhuang}.

%%%%%%%%%%%%%%%%%%%%%%%%%%%%%%%%%%%%%%%%%%%%%%%%%%%%%%%%%%%%%%%%
\subsection{Quasi-morphisms coming from hyperbolic actions}
%%%%%%%%%%%%%%%%%%%%%%%%%%%%%%%%%%%%%%%%%%%%%%%%%%%%%%%%%%%%%%%%
Applying  \cref{thm:qiembedding} together with the Bestvina-Fujiwara \cite{Fujiwara1998,bf2002} construction, we obtain the following. 

\begin{theorem}\label{thm:mainthmwpd}
    Let $G$ be a group that admits a non-elementary action on a geodesic $\delta$-hyperbolic space $X$ with a WPD element. If $S$ is a conjugation-invariant generating set for $G$ such that each $s \in S$ acts elliptically, then $\Z^{m}$ quasi-isometrically embeds into $\Cay(G,S)$ for all $m \in \N$.
\end{theorem} 

In fact this is true even if the $s\in S$ are not elliptic but instead have uniformly bounded translation length, but the elliptic case is already sufficient for applications and so we restrict to that level of generality. 

\begin{corollary} \label{cor:WPDapps}
    For all $m \in \N$, $\Z^{m}$ quasi-isometrically embeds into the Cayley graphs for the following pairs of groups and generating sets:
    \begin{enumerate}[(1)]
        \item The mapping class group of an orientable surface of genus $g$ with $p$ punctures, where $3g+p-3 >1$, with the generating set of all reducible elements.
        \item The fundamental group of a closed surface of genus larger than $2$ with the generating set of all  non-filling curves. 
        \item The outer automorphism group of a free group of rank at least two with the generating set of all reducible outer automorphisms. 
    \end{enumerate}
    In each case, the same result applies for any conjugation invariant generating set contained in the given generating set. 
\end{corollary}

\begin{remark}
The use of quasi-morphisms to show infinite diameter results for groups with infinite generating sets goes back at least to \cite{Rhemtulla}. 
In the case of the fundamental group of a closed surface and a generating set consisting of finitely many mapping class group orbits, Calegari previously used quasi-morphisms to prove the Cayley graph has infinite diameter \cite{Cal}, answering a question of Farb. This infinite diameter result was reproved by Margalit-Putman \cite{MP}. A stronger version of the infinite diameter result was obtained by Brandenbursky-Marcinkowski \cite{bm2019}.

Margalit-Putman  \cite{MP} end their introduction by asking if the Cayley graph of a closed surface group with the generating set of simple closed curves is hyperbolic.   \cref{cor:WPDapps} gives a negative answer to this question.  The part of the proof of \cref{cor:WPDapps} required to confirm this answer consists only of combining  the ideas in \cite{bm2019,MP} and \cite{CancelationNorm}, and also has some overlap with \cite{FW}. 
\end{remark}

%\textcolor{red}{
%\begin{remark}
%    If one is only concerned with a generating set consisting of finitely many conjugacy classes, such as the set of Dehn twists in the mapping class groups, then one can immediately apply \cite[Theorem 3.3]{CancelationNorm} as it is stated together with the fact that these groups admit infinite-dimensional spaces of quasimorphisms \cite{BF2010} to obtain the same result. Alternatively, if one has a generating set consisting of finitely many automorphism group orbits, i.e., the set of simple closed curves in a surface group (\'{a} la the question of Margalit-Putman), then one can directly apply \cite[Theorem 3.3]{CancelationNorm} together with the $\Aut$-invariant quasimorphisms built in \cite{FW}.
%\end{remark}
%}

\begin{remark}
    Note that \Cref{thm:qiembedding} applies in a more general setting than \Cref{thm:mainthmwpd}. In particular, there have been numerous applications of the Bestvina-Fujiwara quasimorphism construction in more general contexts in which one does not have a WPD element. For example: diffeomorphism and homeomorphism groups of surfaces \cite{BHW2022}, big mapping class groups \cite{Bavard2016,HQR2022,Rasmussen2021}, and groups acting on CAT(0) spaces with a rank one element \cite{BF2009,CF2010}. For appropriate generating sets, one should be able to obtain a version of \Cref{cor:WPDapps} for groups of this type. 
\end{remark}

%%%%%%%%%%%%%%%%%%%%%%%%%%%%%%%%%%%%%%%%%%%%%%%%%%%%%%%%%%%%%%%%
\subsection{Free groups}
%%%%%%%%%%%%%%%%%%%%%%%%%%%%%%%%%%%%%%%%%%%%%%%%%%%%%%%%%%%%%%%%

Let $F_n$ denote the free group of  rank $n\geq 2$. Recall that an element of $F_n$ is called separable if it is contained in a proper free factor, and let $\sep_n$ denote the set of separable elements of $F_n$. Applying  \cref{thm:qiembedding} together with counting quasi-morphisms, we obtain the following for the free group $F_n$. 

\begin{theorem}\label{thm:mainthm}
    For all $m\in\N$, $\Z^m$ quasi-isometrically embeds into $\Cay(F_n,\sep_n)$.
\end{theorem} 

\begin{remark}
Fournier-Facio and Wade show that the space of $\Aut(F_n)$-invariant quasi-morphisms on $F_n$ is infinite dimensional  \cite[Theorem B]{FW}, providing an unexpected answer to a question of Abért. One could also use their quasi-morphisms instead of counting quasi-morphisms, although this would be less elementary. In \cite[Theorem E]{FW}, they also produce quasi-morphisms that can be used for certain subsets of the generating sets in \cref{cor:WPDapps}, such as surface groups with the generating set of simple curves and mapping class groups with the generating set of  Dehn twists. 
\end{remark}

\begin{remark}
\Cref{thm:mainthm} recovers in particular that $\Cay(F_n,\sep_n)$ is infinite diameter and non-hyperbolic, facts previously established by Bardakov-Shpilrain-Tolstykh \cite{Palindromic} and Eastham \cite{eastham2024}, respectively. A stronger version of the infinite diameter result was obtained by Brandenbursky-Marcinkowski \cite{bm2019}, and our proof of  \cref{thm:mainthm} again represents a combination of the ideas in \cite{bm2019} and \cite{CancelationNorm}. 
\end{remark}

\begin{remark}
Recall that an element of $F_n$ is called primitive if it is part of a basis, and let $\prim_n$ denote the set of primitive elements of $F_n$. Since primitive elements are separable and every separable element is a product of at most two primitive elements (\cref{rmk:sep-facts}), we have that $\Cay(F_n,\sep_n)$ and $\Cay(F_n,\prim_n)$ are quasi-isometric. 
\end{remark}

\begin{remark}
For any group $G$ and any subgroup $H\subset \Aut(G)$, one can consider generating sets which consist of finitely many $H$-orbits. If $S$ and $S'$ are two such, then the identity map from $\Cay(G,S)$ is $\Cay(G,S')$ is a bi-Lipschitz map and hence in particular a quasi-isometry. In the language of \cite[Section 2.3]{bm2019}, one says that distance to the identity in these Cayley graphs is the $H$-invariant norm on $G$, or more concisely the $H$-norm. Distance to the identity in $\Cay(F_n,\sep_n)$ is the $\Aut(F_n)$-invariant norm on $F_n$.
\end{remark}

%%%%%%%%%%%%%%%%%%%%%%%%%%%%%%%%%%%%%%%%%%%%%%%%%%%%%%%%%%%%%%%%
\subsection{The quotient by conjugation}
%%%%%%%%%%%%%%%%%%%%%%%%%%%%%%%%%%%%%%%%%%%%%%%%%%%%%%%%%%%%%%%%
 Eastham also proves non-hyperbolicity for the quotient of $\Cay(F_n,\sep_n)$ by conjugation, and we can similarly strengthen that result.
% , which she calls the Whitehead complex of $F_n$  

\begin{theorem}\label{thm:whcomplex}
    For all $m \in \N$, $\Z^{m}$ quasi-isometrically embeds into the quotient of $\Cay(F_n,\sep_n)$ by conjugation. 
\end{theorem}

Eastham additionally proves results related to covers, first homology, and the congruence subgroup problem. Her proofs differ from ours and may offer different insights.

%%%%%%%%%%%%%%%%%%%%%%%%%%%%%%%%%%%%%%%%%%%%%%%%%%%%%%%%%%%%%%%%
\subsection{The edge splitting graph}
%%%%%%%%%%%%%%%%%%%%%%%%%%%%%%%%%%%%%%%%%%%%%%%%%%%%%%%%%%%%%%%%
Sabalka-Savchuk obtained related results for the edge splitting graph $\ES_{n}$ of the free group \cite{SS2014}. We also recover this result.

\begin{theorem}[Sabalka-Savchuk]\label{thm:edgesplitting}
    For $n\geq 3$ and all $m \in \N$, $\Z^{m}$ quasi-isometrically embeds into $\ES_{n}$. 
\end{theorem}

Let us now give some context and discussion of this result. The edge splitting graph is also known as the separating sphere graph, and can be viewed as an analogue of the graph $\cC_{\sep}(\Sigma_g)$ of separating curves on a closed surface $\Sigma_g$ of genus at least 3. (As they are defined here, the graphs $\ES_{2}$ and $\cC_{\sep}(\Sigma_2)$ do not have edges, justifying the restrictions to $n, g\geq 3$.) 

The separating curve graph has long been known to be non-hyperbolic \cite[Exercise 2.42]{SaulNotes}. An idea of proof is as follows. The surface $\Sigma_g$ contains a subsurface $U$ such that both $U$ and $V=\Sigma_g-U$ are genus 0 surfaces with $g+1$ boundary components. Every separating curve must cut both $U$ and $V$, so subsurface projections give a Lipschitz map 
 $$\cC_{\sep}(\Sigma_g) \to \cC(U)\times \cC(V).$$
 Lifting a copy of $\bZ^2$ in $\cC(U)\times \cC(V)$ provides quasi-isometrically embedded copies of $\bZ^2$ in $\cC_{\sep}(\Sigma_g)$.

Related ideas, for example in \cite[Conjecture 3.10]{SaulNotes} and \cite{MS}, have since been extensively developed, yielding a beautiful and powerful framework for understanding many subgraphs of curve graphs in the context of hierarchical hyperbolicity \cites{Vokes, Kopreski, CylinderGraph}. In this framework one calls the $U$ and $V$ above \emph{orthogonal witnesses} for $\cC_{\sep}(\Sigma_g)$, and it is known that the maximum dimension of a quasi-flat is equal to the maximum size of a tuple of orthogonal witnesses \cite[Theorem J]{HHS1}. As a result, \cite[Example 2.4]{Vokes} implies that $\cC_{\sep}(\Sigma_g)$ cannot contain a quasi-isometrically embedded copy of $\bZ^3$. 

In hindsight we can thus say that it is not surprising that $\ES_{n}$ is not hyperbolic, but it \emph{is} surprising to us that it contains quasi-isometrically embedded copies of $\bZ^m$ for all $m$. 
Our point of view is that this wild behavior is not caused by many orthogonal witnesses, but rather it is caused by a single  very badly behaved witness, and that this witness is $\Cay(F_{n-1}, \sep_{n-1})$. 

More precisely, one can identify a vertex $g$ in $\Cay(F_{n-1}, \sep_{n-1})$ with factorization $$F_n = F_{n-1} * \langle g a_n\rangle,$$ where $a_1, \ldots, a_n$ is a basis for $F_n$ and we identify $F_{n-1}$ with $\langle a_1, \ldots, a_{n-1}\rangle$. In this way the vertices of $\Cay(F_{n-1}, \sep_{n-1})$ can be identified with vertices in the free splitting graph, and we show in \cref{T:Lip} that after removing a single vertex from the free splitting graph this copy of $\Cay(F_{n-1}, \sep_{n-1})$ becomes a Lipschitz retract. 

It is an interesting problem, which we do not resolve, whether a similar  Lipschitz retract result is possible with the free splitting graph replaced by free factor graph. The vertices of $\Cay(F_{n-1}, \sep_{n-1})$ are  identified with the complex of free factor systems relative to a corank 1 free factor \cite[Section 2.5]{HM}. In the setup of \cite{HM} this is an exceptional case where the complex is zero dimensional, and viewing it as $\Cay(F_{n-1}, \sep_{n-1})$ adds edges in what seems to be a useful way.

%%%%%%%%%%%%%%%%%%%%%%%%%%%%%%%%%%%%%%%%%%%%%%%%%%%%%%%%%%%%%%%%
\subsection{Previous work finding flats of arbitrary dimension}
%%%%%%%%%%%%%%%%%%%%%%%%%%%%%%%%%%%%%%%%%%%%%%%%%%%%%%%%%%%%%%%%
Quasi-flats of arbitrary dimension are known in a wide range of settings. See for example \cites{GG2016,HW2015,Burillo} in the discrete setting and \cites{Brandenbursky2012,CGHS2024,Kloeckner2010,MR2018,Py2008,PS2023,Shapiro2024} in the continuous setting. %Both \cites{GG2016} and \cite{Brandenbursky2012} make use of quasimorphisms.

%%%%%%%%%%%%%%%%%%%%%%%%%%%%%%%%%%%%%%%%%%%%%%%%%%%%%%%%%%%%%%%%
\subsection{Final remarks}
%%%%%%%%%%%%%%%%%%%%%%%%%%%%%%%%%%%%%%%%%%%%%%%%%%%%%%%%%%%%%%%%
%We suggest some questions and possible directions for future work in \cref{sec:questions}. 
%
%Some examples of groups with conjugation invariant generating sets that have previously been studied are listed in \cite[Chapter 10]{Invitation}.

The geometry of the Cayley graphs we study remains mysterious, and might be further studied by investigating what other spaces quasi-isometrically embed into them (compare to \cite[Theorem 3.6]{CancelationNorm} and \cite[Theorem A]{PS2023}) and by investigating their asymptotic cones (compare to \cite[Section 3]{CalegariZhuang} and \cite{HamCone, Karlhofer, Kedra, Directional}).

Much remains to be understood about the retraction $r$ of  \cref{T:Lip}, including expected generalizations to higher corank free factors, if a Bounded Geodesic Image Theorem holds, if it descends to the free factor complex, etc. 

It would be interesting if $\Out(F_n)$ had retractions to Kolchin subgroups (see \cite{Kolchin}), and \cref{T:Lip} may indicate that this would be easier using an appropriate infinite generating set.  

%%%%%%%%%%%%%%%%%%%%%%%%%%%%%%%%%%%%%%%%%%%%%%%%%%%%%%%%%%%%%%%%
\subsection{Acknowledgments}\label{sec:acknowledgements}
%%%%%%%%%%%%%%%%%%%%%%%%%%%%%%%%%%%%%%%%%%%%%%%%%%%%%%%%%%%%%%%%
AW was partially supported by NSF grant DMS-2142712 and a Sloan Research Fellowship. GD was partially supported by NSF grant DMS-2303262. 
AP and SC were supported by the NSF grant of AW, and CG was supported by the Sloan Research Fellowship of AW. 

We especially thank Becky Eastham, whose work inspired ours, for her encouragement and assistance. We thank Michael Brandenbursky,  Micha\l \;Marcinkowski, and Jacob Russell  for their interest and assistance, and the University of Michigan Math Research Experience for Undergraduates program for its support. We thank 
Danny Calegari, Francesco Fournier-Facio, Camille Horbez, Robert Lyman, Dan Margalit, and Ric Wade for helpful conversations. We thank the referee for their helpful comments.

%%%%%%%%%%%%%%%%%%%%%%%%%%%%%%%%%%%%%%%%%%%%%%%%%%%%%%%%%%%%%%%%
%%%%%%%%%%%%%%%%%%%%%%%%%%%%%%%%%%%%%%%%%%%%%%%%%%%%%%%%%%%%%%%%
\section{Free groups}\label{sec:background} 
%%%%%%%%%%%%%%%%%%%%%%%%%%%%%%%%%%%%%%%%%%%%%%%%%%%%%%%%%%%%%%%%
%%%%%%%%%%%%%%%%%%%%%%%%%%%%%%%%%%%%%%%%%%%%%%%%%%%%%%%%%%%%%%%%

%%%%%%%%%%%%%%%%%%%%%%%%%%%%%%%%%%%%%%%%%%%%%%%%%%%%%%%%%%%%%%%%
\subsection{Background}
%%%%%%%%%%%%%%%%%%%%%%%%%%%%%%%%%%%%%%%%%%%%%%%%%%%%%%%%%%%%%%%%
Let $F_n$ be the free group with basis $a_1, \ldots, a_n$. We always assume $n\geq 2$. A subgroup of $F_n$ is called a free factor if it is generated by a subset of a basis, and we call a free factor proper if it is not all of $F_n$. Recall that an element of $F_n$ is called separable if it is contained in a proper free factor, and is called primitive if it is part of a basis. 
\begin{remark}\label{rmk:sep-facts} We have the following observations.
\begin{enumerate}
    \item Primitive elements are separable.
    \item Powers of separable elements are separable.
    \item The set of separable elements is invariant under $\Aut(F_n)$.
    \item In particular, the set of separable elements is invariant under conjugation. 
    \item If $\langle y_1,\ldots, y_n\rangle$ is a basis of $F_n$ and $a\in\langle y_1,\ldots,y_{n-1}\rangle$, then $\langle y_1, \ldots, y_{n-1}, ay_n\rangle$ is a basis, so $ay_n$ is primitive. 
    \item Every separable element can be written as the product of two primitive elements, since if $a$ is separable there exists a basis $ y_1, \ldots, y_n$ with $a\in\langle y_1,\ldots,y_{n-1}\rangle$, and $a=(ay_n)(\overline{y}_n)$.
\end{enumerate}
\end{remark}
Note that we sometimes use $\overline{g}$ to denote the inverse $g^{-1}$ of a group element $g$. 

\begin{definition}\label{D:W}
    Let $w=l_1\ldots l_k$ be a cyclically reduced word in $F_n$. The Whitehead graph $\Omega(w)$ of $w$ is defined as follows. 
    \begin{enumerate}
    \item There are $2n$ vertices, labeled $a_1, \ldots , a_n,\overline{a}_1,\ldots, \overline{a}_n$.
    \item For every pair of adjacent letters $l_i l_{i+1}$ in $w$, add an edge connecting vertex $l_i$ to $\overline{l}_{i+1}$.
    \item\label{D:W:3}Then add an edge joining vertices $l_k$ and $\overline{l}_1$. We call this the wrap around edge.
    \item When there are multiple edges joining the same pair of vertices, replace them with a single edge.
    \end{enumerate}
    The Whitehead graph of a word that is not cyclically reduced is defined to be the Whitehead graph of its cyclic reduction. 
\end{definition} 

\begin{definition}
    We denote by $\cut_n$ the set of $w\in F_n$ for which $\Omega(w)$ has a cut vertex.
\end{definition}

Here a vertex is a cut vertex if after removing it and any edges adjacent to it the graph is disconnected. Note that, in our definition of $\Omega(w)$, there cannot be multiple edges joining the same pair of vertices. There also cannot be loops from a vertex to itself, because the word is cyclically reduced. Our standing assumption $n\geq 2$ implies that $\Omega(w)$ always has at least 4 vertices, so if $\Omega(w)$ is disconnected then it automatically has a cut vertex.

For a reference on the following, see  \cite{Whitehead1936,stallings1999,WhiteheadShort}. 

\begin{theorem}[Whitehead's Cut Vertex Lemma]
    The Whitehead graph of a separable element has a cut vertex. That is, $\sep_n \subset \cut_n$.
\end{theorem}

We also make use of a slight modification of the Whitehead graph. 
\begin{definition}
    Let $w=l_1\ldots l_k$ be a reduced word in $F_n$. The graph $\Omega'(w)$ is defined as in \cref{D:W} but omitting step \eqref{D:W:3}.
 Let $\cut_n'$ be the set of $w\in F_n$ for which $\Omega'(w)$ has a cut vertex. 
\end{definition} 
We emphasize that, in forming $\Omega'(w)$, we do not cyclically reduce $w$ and we do not add a wrap around edge. See \Cref{fig:WHexamples} for an example of this distinction.

\begin{figure}[!hbt]
      \begin{tikzpicture}
      \tikzstyle{every node}=[font=\large]
      \node (Wh) at (1,5) {$\Omega(w)$};
      \node (a1) at (0,4) {$a_1$};
      \node (a2) at (0,3) {$a_2$};
      \node (a3) at (0,2) {$a_3$};
      \node (a4) at (0,1) {$a_4$};
      \node (a1bar) at (2, 4) {$\overline{a}_1$};
      \node (a2bar) at (2,3) {$\overline{a}_2$};
      \node (a3bar) at (2,2) {$\overline{a}_3$};
      \node (a4bar) at (2,1) {$\overline{a}_4$};

      \graph{
      (a2) -- (a2bar);
      (a2) -- (a3bar);
      (a3) -- (a3bar);
      (a3) -- (a4bar);
      (a4) -- (a4bar);
      };
      \draw (a2bar) .. controls (-1.5,5) and (-1,1) .. (a4);

      \node (Wh') at (6,5) {$\Omega'(w)$};
      \node (b1) at (5,4) {$a_1$};
      \node (b2) at (5,3) {$a_2$};
      \node (b3) at (5,2) {$a_3$};
      \node (b4) at (5,1) {$a_4$};
      \node (b1bar) at (7, 4) {$\overline{a}_1$};
      \node (b2bar) at (7,3) {$\overline{a}_2$};
      \node (b3bar) at (7,2) {$\overline{a}_3$};
      \node (b4bar) at (7,1) {$\overline{a}_4$};

      \graph{
      (b1) -- (b2bar);
      (b2) -- (b2bar);
      (b2) -- (b3bar);
      (b3) -- (b3bar);
      (b3) -- (b4bar);
      (b4) -- (b4bar);
      };
      \draw (b1) .. controls (4,3) and (4,2) .. (b4);

      \end{tikzpicture}

     \caption{An illustration of the difference between $\Omega(w)$ and $\Omega'(w)$ for $w = a_1 a_2^2 a_3^2 a_4^2 \overline{a}_1 \in F_{4}$.}
     \label{fig:WHexamples}
 \end{figure}

\begin{remark}\label{R:cut'}
If $w$ is a  subword of a cyclically reduced word $u$, we have $\Omega'(w)\subset \Omega(u)$. So if $u\in \cut_n$ is cyclically reduced, then every subword of $u$ is in $\cut'_n$. 
\end{remark}

%%%%%%%%%%%%%%%%%%%%%%%%%%%%%%%%%%%%%%%%%%%%%%%%%%%%%%%%%%%%%%%%
\subsection{Many quasi-morphisms}\label{SS:Brooks}
%%%%%%%%%%%%%%%%%%%%%%%%%%%%%%%%%%%%%%%%%%%%%%%%%%%%%%%%%%%%%%%%
We will now see that Whitehead's Lemma implies the following, which is a strengthening of \cite[Lemma 3.4]{bm2019}.

\begin{proposition}\label{P:Fnquasis}
There is an infinite dimensional space of homogeneous quasi-morphisms on $F_n$ that vanish on $\cut_n$. 
\end{proposition}

In light of  \cref{thm:qiembedding}, this immediately implies  \cref{thm:mainthm}. 

\begin{customthm}{1.6}
    For all $m \in \N$, $\Z^m$ quasi-isometrically embeds into $\Cay(F_n,\sep_n)$. 
\end{customthm}

\begin{proof}
    \Cref{P:Fnquasis} provides an infinite dimensional space of quasi-morphisms that vanish on $\sep_n$, since $\sup_n \subset \cut_n$. Finally, $\sep_n$ is a conjugation invariant generating set for $F_n$ and so we can apply \Cref{thm:qiembedding} with this collection of quasi-morphisms. 
\end{proof}

Given $w\in F_n$, recall that the \textit{counting quasimorphism} $q_w:F_n\to\R$ is defined as follows \cite{brooks1981}. For a reduced word $x\in F_n$, let $C_w(x)$ and $C_{\overline{w}}(x)$ count occurrences of $w$ and $\overline{w}$ as subwords in $x$, with overlap allowed. Then $q_w(x)=C_w(x)-C_{\overline{w}}(x)$.

For each $k \in \N$ define
\begin{align*}
    p_{k} = a_{1}^{k+1}a_{2}^{k+1}\cdots a_{n}^{k+1}a_{1} \in F_n,
\end{align*}
and let $q_{k}:F_{n} \rightarrow \R$ denote the counting quasimorphism for $w = p_{k}$. Let $\Tilde{q}_k$ denote the homogenization of $q_k$, so 
$$\Tilde{q}_k(s)=\lim_{r\to\infty}\frac{q_k(s^r)}{r}$$
for all $s\in F_n$. 
Observe that 
$$q_k(p_{j}^r) = r \delta_{k,j},$$ 
and hence 
$$\Tilde{q}_k(p_{j}) = \delta_{k,j}.$$
In particular, the $\Tilde{q}_k, k\in \N$ are linearly independent. Thus the following is enough to conclude the proof of \cref{P:Fnquasis}.

\begin{lemma}\label{lm:qmsep0}
    For any  $s\in \cut_n$, $\tilde{q}_k(s)=0$.
\end{lemma}

\begin{proof}
Note that $\Omega(p_k)=\Omega'(p_k)$ (illustrated in \Cref{fig:pk-whitehead}) is connected with no cut vertices. 
     \begin{figure}[!hbt]
      \begin{tikzpicture}
      \tikzstyle{every node}=[font=\large]
      \node (a1) at (0,4) {$a_1$};
      \node (a2) at (0,3) {$a_2$};
      \node (vdots1) at (0,2) {$\vdots$};
      \node (an1) at (0,0) {$a_{n-1}$};
      \node (an) at (0,-1) {$a_n$};
      \node (a1bar) at (2, 4) {$\overline{a}_1$};
      \node (a2bar) at (2,3) {$\overline{a}_2$};
      \node (vdots2) at (2,1.25) {$\vdots$};
      \node (an1bar) at (2,0) {$\overline{a}_{n-1}$};
      \node (anbar) at (2,-1) {$\overline{a}_n$};
      \graph{
      (a1) -- (a1bar);
      (a1) -- (a2bar);
      (a2) -- (a2bar);
      (an1) -- (an1bar);
      (an1) -- (anbar);
      (an) -- (anbar);
      };
      \draw (a2) -- (2,2);
      \draw (an1bar) -- (0,1);
      \draw (a1bar) .. controls (-2,6) and (-1.5,1) .. (an);
      \end{tikzpicture}

     \caption{An illustration of $\Omega(p_k)=\Omega'(p_k)$}
     \label{fig:pk-whitehead}
 \end{figure}
    Thus  $p_k$ cannot be contained in any cyclically reduced word with a cut vertex (\cref{R:cut'}), so for any such $s$, we have $q_k(s)=0$. As $s^n$ is also cyclically reduced and has the same Whitehead graph, this implies that 
    $\Tilde{q}_k(s)=0.$
\end{proof}

\begin{remark}
\cite[Theorem 3.6]{bm2019} shows that for every non-separable element of $F_n$ there is a homogeneous quasi-morphism which is non-zero on that element but is bounded on the set of all primitive elements (and hence is also bounded on the set of all separable elements). 
\end{remark}

%%%%%%%%%%%%%%%%%%%%%%%%%%%%%%%%%%%%%%%%%%%%%%%%%%%%%%%%%%%%%%%%
\subsection{The quotient by conjugation}
%%%%%%%%%%%%%%%%%%%%%%%%%%%%%%%%%%%%%%%%%%%%%%%%%%%%%%%%%%%%%%%%
We now pause to note that \cref{P:Fnquasis} also implies \cref{thm:whcomplex}, by noting a generalization of \cref{thm:qiembedding}. 

Let $S$ be a conjugation invariant generating set for a group $G$. The action of $G$ on itself by conjugation induces an action of $G$ on $\Cay(G, S)$ by graph automorphisms. We consider the quotient graph $\Cay(G, S)/\Inn(G)$. Two elements of the quotient graph are joined by an edge if and only if they have lifts that are joined by an edge. The fact that the action is by graph automorphisms implies that any path in the quotient can be lifted, so we conclude that if $g,h\in G$, then $$d_{\Cay(G, S)/\Inn(G)}([g],[h])=\min_{k\in G} d_{\Cay(G, S)}(g,\overline{k}hk).$$
Now we observe that if $q$ is any homogeneous quasi-morphism, then 
$$q(g (\overline{k}hk)^{-1}) = q(g h^{-1})+O(1),$$
where the $O(1)$ term depends only on the defect. 
Adding this observation to the proof of \cref{thm:qiembedding} gives the following variant.

\begin{theorem}\label{thm:qiembeddingvariant}
Let $S$ be a conjugation invariant generating set for a group $G$. Suppose $G$ has $m$ linearly independent homogeneous quasi-morphisms which are bounded on $S$. Then there is a quasi-isometric embedding of $\bZ^m$ into $\Cay(G,S)/\Inn(G)$. 
\end{theorem}

In light of this result, \cref{P:Fnquasis} immediately implies \cref{thm:whcomplex}.

\begin{customthm}{1.11}
    For all $m \in \N$, $\Z^m$ quasi-isometrically embeds into the quotient of $\Cay(F_n,\sep_n)$ by conjugation.
\end{customthm}

\begin{proof}
    \Cref{P:Fnquasis} provides an infinite dimensional space of homogenous quasimorphisms that vanish on $\cut_n$ and hence on $\sep_n$. Thus applying \Cref{thm:qiembeddingvariant} gives the desired result. 
\end{proof}

%%%%%%%%%%%%%%%%%%%%%%%%%%%%%%%%%%%%%%%%%%%%%%%%%%%%%%%%%%%%%%%%
%%%%%%%%%%%%%%%%%%%%%%%%%%%%%%%%%%%%%%%%%%%%%%%%%%%%%%%%%%%%%%%%
\section{The free-splitting graph}\label{sec:free-splitting}
%%%%%%%%%%%%%%%%%%%%%%%%%%%%%%%%%%%%%%%%%%%%%%%%%%%%%%%%%%%%%%%%
%%%%%%%%%%%%%%%%%%%%%%%%%%%%%%%%%%%%%%%%%%%%%%%%%%%%%%%%%%%%%%%%

In this section we assume $n\geq 3$ and require familiarity with the basics of Bass-Serre Theory. 

%%%%%%%%%%%%%%%%%%%%%%%%%%%%%%%%%%%%%%%%%%%%%%%%%%%%%%%%%%%%%%%%
\subsection{Statement of the main result.}
%%%%%%%%%%%%%%%%%%%%%%%%%%%%%%%%%%%%%%%%%%%%%%%%%%%%%%%%%%%%%%%%
Let $FS_n$ denote the free-splitting graph, or in other words the 1-skeleton of the free splitting complex. See for example \cite{HMhyp} for an introduction. Recall that minimal actions of $F_n$ on trees with trivial edge stabilizers give vertices of $FS_n$, and edges of $FS_n$ come from collapse maps. Recall  also that by definition a $k$-vertex splitting is one with $k$ orbits of vertices, and a $k$-edge splitting is one with $k$ orbits of edges. 

Fix a basis $a_1, \ldots, a_n$ of $F_n$, and let $H=\langle a_1, \ldots, a_{n-1}\rangle \cong F_{n-1}$. There is a unique (up to conjugacy) 1-edge 1-vertex free splitting where $H$ fixes a vertex, and we denote this $T_H$. For $w\in H$, there is a unique (up to conjugacy) 1-edge 2-vertex free splitting where $H$ fixes a vertex and $\langle wa_n\rangle$  fixes a vertex, and we denote this $T_{H, w}$. (These facts follow from Bass-Serre Theory; see in particular \cite[Theorem 13]{Trees}.) The main result of this section is the following. 

\begin{theorem}\label{T:Lip}
There is a Lipschitz map 
$$r: FS_n \setminus \{T_H\} \to \Cay(H, \sep_H)$$
such that $r(T_{H,w})=w$. 
\end{theorem}

Here $\sep_H$ denotes the separable elements of $H$. Variants of \cref{T:Lip} likely hold when $H$ is a higher co-rank free factor, but we will not pursue them here. Speaking imprecisely, $r$ might be viewed as one way to measure ``folding over $H$''.

One can view $r$ as analogous to a subsurface projection. For other analogues of subsurface projection in the context of $\Out(F_n)$ and other Lipschitz maps, see \cite{SSproj, BFproj, TaylorProj, RelTw,Gupta, HMretract, HH}. 

%%%%%%%%%%%%%%%%%%%%%%%%%%%%%%%%%%%%%%%%%%%%%%%%%%%%%%%%%%%%%%%%
\subsection{The edge-splitting graph.}
%%%%%%%%%%%%%%%%%%%%%%%%%%%%%%%%%%%%%%%%%%%%%%%%%%%%%%%%%%%%%%%%
The edge splitting graph $ES_n$ is defined in \cite{SS2014} to be the graph with vertices corresponding to conjugacy classes of (unordered) pairs $(A,B)$  of non-trivial subgroups with $F_n=A\ast B$. If $F_n=A\ast B \ast C$, one adds an edge from $(A\ast B, C)$ to $(A, B\ast C)$.  

Given such a pair $(A,B)$, there is an associated 1-edge 2-vertex free splitting where $A$ and $B$ are both vertex stabilizers \cite[Theorem 7]{Trees}. We denote this $T_{A,B}$. In this way the vertices of $ES_n$ determine vertices of $FS_n$. If there is an edge between two vertices of $ES_n$ then the corresponding vertices of $FS_n$ are both adjacent to a 2-edge 3-vertex splitting. In this way we see that $ES_n$ has a Lipschitz map into $FS_n\setminus\{T_H\}$, which we denote 
$$\tau_{BS} : ES_n \to FS_n\setminus\{T_H\}$$
since it is associated with Bass-Serre Theory. 

There is also a map $$i: \Cay(H, \sep_H)\to ES_n$$ defined by $i(w) = (H, \langle wa_n\rangle)$. Note that if $u\in \sep_H$ we can write $H=H_1\ast H_2$ with $u\in H_2$, and both $i(w)$ and $i(uw)$ are adjacent to $(H_1, H_2\ast \langle w a_n\rangle )$. In particular, we see that $i$ is Lipschitz. 

Theorem \ref{T:Lip} gives that $r \circ \tau_{BS} \circ i$ is the identity. From this we get that $i$ must be a quasi-isometric embedding. Since $ES_n$ contains a quasi-isometrically embedded copy of $\Cay(H, \sep_H)$,  \cref{thm:mainthm} now implies \cref{thm:edgesplitting}. 

%%%%%%%%%%%%%%%%%%%%%%%%%%%%%%%%%%%%%%%%%%%%%%%%%%%%%%%%%%%%%%%%
\subsection{Proof of \Cref{T:Lip}}
%%%%%%%%%%%%%%%%%%%%%%%%%%%%%%%%%%%%%%%%%%%%%%%%%%%%%%%%%%%%%%%%
We start by recalling some basic facts, which we will use below without further comment. We say an action is elliptic if it has a bounded orbit. 

\begin{remark}\label{R:SplittingFacts}
Suppose that $T$ is a free splitting. 
\begin{enumerate}
\item Any subgroup $K$ of $F_n$ has a unique subtree $\Core_K(T)$ of $T$ which is $K$-invariant and on which the $K$ action is minimal \cite[Theorem 3.1]{CM}. If $K$ is finitely generated, it has finitely many $K$-orbits of edges.  
\item If $K\neq \{\id\}$ acts elliptically it has a unique fixed point (because edge stabilizers are trivial) and $\Core_K(T)$ consists of this point. 
\item If a group element $g$ is not elliptic it is loxodromic \cite[Section 1.3]{CM}. In this case $\Core_{\langle g \rangle}(T)$ is an oriented real line, and is also denoted $\Axis(g)$. 
\item\label{R:Hell} If the corank 1 free factor $H$ acts elliptically then $T$ is conjugate to $T_H$ or one of the $T_{H,w}$. (The associated graph of groups is easily determined and determines $T$  \cite[Theorem 13]{Trees}.)
\item\label{R:FiniteIntersection} If $a_n$ is loxodromic, its axis $\Axis(a_n)$ intersects $\Core_H(T)$ in at most finitely many edges (since $\Axis(a_n)$ has finitely many $\langle a_n\rangle$-orbits of edges, $\Core_H(T)$ has finitely many $H$-orbits of edges, and since edge stabilizers are trivial). 
\end{enumerate}
\end{remark}

We define $r(T)$ in four cases. 

\textbf{Case 1: $H$ elliptic, $a_n$ elliptic.} In this case $T=T_{H,a_n}$ and we define $r(T)=\id$ to be the identity element of $H$.  

\textbf{Case 2: $H$ elliptic, $a_n$ not elliptic.} Let $v_H\in T$ denote the elliptic fixed point of $T$. If $\Axis(a_n)$ does not contain $v_H$, define $$r(T)=\id.$$ 

Otherwise, let  $\entrye$ be the edge along which $\Axis(a_n)$ enters the vertex $v_H$ and let $\exite$ be the edge along which $\Axis(a_n)$  leaves $v_H$. So both $\entrye$ and $\exite$ are adjacent to $v_H$. Since $T\neq T_H$, there is only one $H$-orbit of edges adjacent to $v_H$, so  there is a unique $h\in H$ with $h\exite  =\entrye$. We define $r(T)=h$. Thus we have 
$$r(T) \exite  =\entrye.$$

\textbf{Case 3: $H$ not elliptic, $a_n$ elliptic.} In this case we define 
$$r(T)=\id.$$

\textbf{Case 4: $H$ not elliptic, $a_n$ not elliptic.} If $\Axis(a_n)$ is disjoint from $\Core_H(T)$ or these sets intersect in a single point, we define $$r(T)=\id.$$

Otherwise, let $\entryv$ be the vertex at which $\Axis(a_n)$ enters $\Core_H(T)$ and let $\exitv$ be the vertex at which it exits $\Core_H(T)$. 

Let $D_{\entryv}$ be a connected union of edges of $T_H$ that has exactly one edge from each $H$ orbit and contains $\entryv\in D_{\entryv}$. We think of $D_{\entryv}$ as a fundamental domain, but one should keep in mind that vertex stabilizers are typically non-trivial so some caution is required.

We define $r(T)$ to be any choice of element in $H$ such that 
$$r(T) \exitv \in D_{\entryv}.$$
Speaking roughly, this represents an approximate equality $r(T) \exitv \approx \entryv.$

Note that this means our $r$ is not canonically defined, since we have chosen a $D_{\entryv}$  and then chosen this element. However we have the following. 

\begin{lemma}\label{L:AlmostWellDefined}
If $h$ and $h'$ are two different possible definitions of $r(T)$ in Case 4, then $h$ and $h'$ are uniformly close depending only on $n$ in $\Cay(H, \sep_H)$. 
\end{lemma}

\begin{proof}
Consider fundamental domains $D_{\entryv}$ and $D_{\entryv}'$, as above, both containing $\entryv$ with $h \exitv \in D_{\entryv}$ and $h' \exitv \in D_{\entryv}'$. We get 
$$D_{\entryv} \cap h (h')^{-1}D'_{\entryv} \neq \emptyset.$$ 

Each of these domains has at most a uniformly finite number of edges, depending only on the rank of the free group. This bound comes from the number of edges in a trivalent graph with fundamental group $F_n$. It follows that there is a uniform bound, depending only on $n$, for how far $h (h')^{-1}$ moves $\entryv$. Thus we can find a uniformly finite chain of translates of $D_{\entryv}$, each sharing a vertex with the previous one, such that $\entryv$ is contained in the first and $h (h')^{-1} \entryv$ is contained in the last. As is implicit in Bass-Serre Theory and recalled in the next lemma, if $g_1 D_{\entryv}$ is adjacent to $D_{\entryv}$ then $g_1$ is separable or a product of two separables. This gives the result.
\end{proof}

\begin{lemma}
Let $H$ be a free group acting minimally on a tree $S$ with trivial edge stabilizers. Let $D$ be a connected subtree containing one edge from each orbit of edges. If $h\in H$ is such that $hD$ is adjacent to $D$, then $h$ is the product of at most two separables. 
\end{lemma}

\begin{proof}
Consider a subtree $D'\subset D$ which maps to a spanning subtree of the quotient $D/H$. We can assume that $D'$ contains the edge of $D$ adjacent to $hD$. Let $v$ be the vertex where $D$ and $hD$ are adjacent. 

Bass-Serre Theory gives that $F$ is the free product of vertex stabilizers of vertices in $D'$, together with one free generator $t_e$ for each edge $e$ in $D$ not in $D'$. This $t_e$ maps the endpoint of $e$ not in $D'$ into $D'$. 

Consider the edges in $D'$ together with the edges $t_e(e)$. If one considers only those edges in this collection adjacent to $v$, this has exactly one representative of each orbit of the stabilizer of $v$ acting on the set of edges adjacent to $v$. It follows that $hD$ can be obtained from $D$ by first possibly applying a $t_e$ and then possibly applying an element of the vertex stabilizer. This gives the result. 
\end{proof}

The most important part of the proof of \cref{T:Lip} is the following. 

\begin{lemma}
 $r$ is Lipschitz.
 \end{lemma}

 \begin{proof}
 It suffices to consider a collapse map $c: T \to T'$ that collapses a single orbit of edges, and give a uniform upper bound for the distance between $r(T)$ and $r(T')$. 

We will speak of $T$ and $T'$ having type 1, 2, 3, or 4, according to which of the 4 cases above it falls under. We recall these in the following table. 
\begin{center}
\begin{tabular}{l|l|l|l} % 'c' for centered text, '|' for vertical lines
    \multicolumn{4}{c}{} \\ % Empty row for top spacing if desired
    %\hline
    & $H$ elliptic & $a_n$ elliptic  &  definition of $r$ \\
    \hline
    1& yes & yes & $r(T)=\id$ \\
    \hline
    2 & yes & no & $r(T)=\id$ or $r(T) \exite  =\entrye$ \\
    \hline
    3 & no & yes & $r(T)=\id$ \\
    \hline
    4 & no & no & $r(T)=\id$ or $r(T) \exitv \in D_{\entryv}$ \\
    \multicolumn{4}{c}{} \\ % Empty row for bottom spacing if desired
\end{tabular}
\end{center}
Note that $T$ must be type 3 or 4, since type 1 and 2 splittings have only one orbit of edges. Additionally, we note that if $a_n$ is elliptic for $T$ it must be elliptic for $T'$. It thus suffices to consider the following cases. In all cases we keep in mind that $$c(\Core_K(T)) = \Core_K(T')$$ for any subgroup $K$, and that this  specializes to a number of useful statements when $K$ is $H$ or $\langle a_n\rangle$.   

\textbf{$T$ type 3 collapses to $T'$ type 1 or type 3:} This is trivial since $r(T)=\id=r(T')$. 

\textbf{$T$ type 4 collapses to $T'$ type 1:} Because the fixed point of $a_n$ in $T'$ is not in $\Core_H(T')$ we see that the axis of $a_n$ in $T$ is disjoint from $\Core_H(T)$, so we have $r(T)=\id=r(T')$. 

\textbf{$T$ type 4 collapses to $T'$ type 2:} 
In this case the axis of $a_n$ in $T$ maps to the axis of $a_n$ in $T'$ under $c$. After the collapse, say $h\in H$ maps the exit edge to the entry edge. Before the collapse $h$ maps an edge on the axis after entry to an edge on the axis before entry. So it exactly maps the exit vertex to the entry vertex.

\textbf{$T$ type 4 collapses to $T'$ type 3:} We have $r(T')=\id$. If the axis of $a_n$ in $T$ is disjoint from $\Core_H(T)$ we also have $r(T)=\id$, so assume this is not the case. Since $r(T) \exitv \in D_{\entryv}$, applying $c$ we get 
$$r(T) c(\exitv) \in c(D_{\entryv}).$$
Since $a_n$ becomes elliptic in $T'$, its axis is mapped to a point, so in particular $c(\entryv)=c(\exitv)$. Since $c(D_{\entryv})$ contains uniformly finitely many edges, we get that $r(T)$ moves a point only a uniformly small distance. Namely, they are only moved at most the number of edges in $c(D_{\entryv})$. Thus (as in \cref{L:AlmostWellDefined}) we get that $r(T)$ is a product of a bounded number of separable elements. 

\textbf{$T$ type 4 collapses to $T'$ type 4:} If the axis of $a_n$ in $T$ is disjoint from $\Core_H(T)$, then the axis of $a_n$ in $T'$ intersects $\Core_H(T')$ in at most a single point, so we get that both $r(T)$ and $r(T')$ are the identity. In the remaining case it suffices to notice that the entry and exit vertices of the axis in $T$ map to the entry and exit vertices in $T'$. 

This concludes the proof that $r$ is Lipschitz. 
\end{proof}

To conclude the proof of \cref{T:Lip}, it now suffices to compute $r(T_{H,w})$. Let $E$ be the edge with one vertex stabilizer $H$ and the other $\langle w a_n\rangle$. We claim that $E \cup wE$ is a fundamental domain for the axis of $a_n$, if $w$ is not the identity. 
% DO NOT DELETE (AW Jun 3 3025): 
%(One way to see this is to note that $a_n$ cannot move an edge to an adjacent edge, so the minimal amount a midpoint of an edge can be moved is 2. Another way is to note that for an edge not on the axis, it's orientation relation to its image is different.)

\begin{figure}[h!]

\tikzset{every picture/.style={line width=0.75pt}} %set default line width to 0.75pt        

\begin{tikzpicture}[x=0.75pt,y=0.75pt,yscale=-1,xscale=1]
%uncomment if require: \path (0,300); %set diagram left start at 0, and has height of 300

%Straight Lines [id:da12255084324152321] 
\draw    (99.86,120) -- (209.86,120) ;
\draw [shift={(209.86,120)}, rotate = 0] [color={rgb, 255:red, 0; green, 0; blue, 0 }  ][fill={rgb, 255:red, 0; green, 0; blue, 0 }  ][line width=0.75]      (0, 0) circle [x radius= 3.35, y radius= 3.35]   ;
\draw [shift={(99.86,120)}, rotate = 0] [color={rgb, 255:red, 0; green, 0; blue, 0 }  ][fill={rgb, 255:red, 0; green, 0; blue, 0 }  ][line width=0.75]      (0, 0) circle [x radius= 3.35, y radius= 3.35]   ;
%Straight Lines [id:da14309158202665573] 
\draw    (209.86,120) -- (319.86,120) ;
\draw [shift={(319.86,120)}, rotate = 0] [color={rgb, 255:red, 0; green, 0; blue, 0 }  ][fill={rgb, 255:red, 0; green, 0; blue, 0 }  ][line width=0.75]      (0, 0) circle [x radius= 3.35, y radius= 3.35]   ;
\draw [shift={(209.86,120)}, rotate = 0] [color={rgb, 255:red, 0; green, 0; blue, 0 }  ][fill={rgb, 255:red, 0; green, 0; blue, 0 }  ][line width=0.75]      (0, 0) circle [x radius= 3.35, y radius= 3.35]   ;
%Straight Lines [id:da512189840607104] 
\draw    (319.86,120) -- (429.86,120) ;
\draw [shift={(429.86,120)}, rotate = 0] [color={rgb, 255:red, 0; green, 0; blue, 0 }  ][fill={rgb, 255:red, 0; green, 0; blue, 0 }  ][line width=0.75]      (0, 0) circle [x radius= 3.35, y radius= 3.35]   ;
\draw [shift={(319.86,120)}, rotate = 0] [color={rgb, 255:red, 0; green, 0; blue, 0 }  ][fill={rgb, 255:red, 0; green, 0; blue, 0 }  ][line width=0.75]      (0, 0) circle [x radius= 3.35, y radius= 3.35]   ;
%Curve Lines [id:da5020326265391123] 
\draw [color={rgb, 255:red, 155; green, 155; blue, 155 }  ,draw opacity=1 ]   (153,96) .. controls (173.44,77.38) and (228.87,77.01) .. (252.45,94.88) ;
\draw [shift={(253.86,96)}, rotate = 219.91] [color={rgb, 255:red, 155; green, 155; blue, 155 }  ,draw opacity=1 ][line width=0.75]    (10.93,-3.29) .. controls (6.95,-1.4) and (3.31,-0.3) .. (0,0) .. controls (3.31,0.3) and (6.95,1.4) .. (10.93,3.29)   ;
%Curve Lines [id:da4186162782686583] 
\draw [color={rgb, 255:red, 155; green, 155; blue, 155 }  ,draw opacity=1 ]   (270,96) .. controls (290.44,77.38) and (345.87,77.01) .. (369.45,94.88) ;
\draw [shift={(370.86,96)}, rotate = 219.91] [color={rgb, 255:red, 155; green, 155; blue, 155 }  ,draw opacity=1 ][line width=0.75]    (10.93,-3.29) .. controls (6.95,-1.4) and (3.31,-0.3) .. (0,0) .. controls (3.31,0.3) and (6.95,1.4) .. (10.93,3.29)   ;

% Text Node
\draw (255,102.4) node [anchor=north west][inner sep=0.75pt]    {$E$};
% Text Node
\draw (139,102.4) node [anchor=north west][inner sep=0.75pt]    {$wE$};
% Text Node
\draw (353,100.4) node [anchor=north west][inner sep=0.75pt]    {$wa_{n} E$};
% Text Node
\draw (191,62.4) node [anchor=north west][inner sep=0.75pt]  [color={rgb, 255:red, 155; green, 155; blue, 155 }  ,opacity=1 ]  {$w^{-1}$};
% Text Node
\draw (305,62.4) node [anchor=north west][inner sep=0.75pt]  [color={rgb, 255:red, 155; green, 155; blue, 155 }  ,opacity=1 ]  {$wa_{n}$};
% Text Node
\draw (203,128.4) node [anchor=north west][inner sep=0.75pt]    {$H$};
% Text Node
\draw (300,128.4) node [anchor=north west][inner sep=0.75pt]    {$\langle wa_{n} \rangle $};

\end{tikzpicture}

\caption{A fundamental domain for the axis of $a_n$ in the splitting $H \ast \langle wa_n\rangle$.}
\label{F:2nhood}
\end{figure}  

The axis for $a_n$ enters the $H$ fixed point via $wE$ and leaves via $E$, so $w$ is the unique element that sends the exit edge to the entry edge.

%%%%%%%%%%%%%%%%%%%%%%%%%%%%%%%%%%%%%%%%%%%%%%%%%%%%%%%%%%%%%%%%
%%%%%%%%%%%%%%%%%%%%%%%%%%%%%%%%%%%%%%%%%%%%%%%%%%%%%%%%%%%%%%%%
\section{Quasimorphism from hyperbolic actions} \label{sec:generalqm}
%%%%%%%%%%%%%%%%%%%%%%%%%%%%%%%%%%%%%%%%%%%%%%%%%%%%%%%%%%%%%%%%
%%%%%%%%%%%%%%%%%%%%%%%%%%%%%%%%%%%%%%%%%%%%%%%%%%%%%%%%%%%%%%%%

In this section we will sketch how to apply the constructions of Fujiwara \cite{Fujiwara1998} and Bestvina-Fujiwara \cite{bf2002} to groups that admit ``nice'' actions on hyperbolic spaces in order to construct quasimorphisms to which \Cref{thm:qiembedding} applies.

\begin{theorem}[Bestvina-Fujiwara]\label{thm:bf}
    Let $G$ be a group that admits a non-elementary action on a geodesic $\delta$-hyperbolic space $X$ with a WPD element. Then there exists an infinite sequence of linearly independent homogeneous quasi-morphisms, each bounded on elliptic elements of the action. 
\end{theorem}

Using this we can immediately apply \Cref{thm:qiembedding} to obtain \Cref{thm:mainthmwpd}. This theorem does not appear in this exact form in \cite{bf2002}, but it follows readily from their constructions. In \Cref{SS:BF} we will give a brief outline of the proof. We first discuss the applications. %%% mentioned in the introduction. 

%%%%%%%%%%%%%%%%%%%%%%%%%%%%%%%%%%%%%%%%%%%%%%%%%%%%%%%%%%%%%%%%
\subsection{Mapping Class Groups}
%%%%%%%%%%%%%%%%%%%%%%%%%%%%%%%%%%%%%%%%%%%%%%%%%%%%%%%%%%%%%%%%
    Let $\Sigma = \Sigma_{g,p}$ be an orientable surface of genus $g$ with $p$ punctures. The complexity of $\Sigma_{g,p}$ is given by $\kappa(\Sigma) = 3g+p-3$. The \emph{mapping class group} of $\Sigma$, denoted $\MCG(\Sigma)$, is the group of isotopy classes of orientation preserving homeomorphisms of $\Sigma$. The \emph{curve graph} of $\Sigma$, denoted $\mathcal{C}(\Sigma)$, is the graph with vertices consisting of isotopy classes of essential, non-peripheral simple closed curves and edges denoting disjointness up to isotopy. 
    
    \begin{corollary}
        Let $\Sigma$ be an orientable surface with $\kappa(\Sigma) > 1$. For all $m \in \N$, $\Z^{m}$ quasi-isometrically embeds into the Cayley graph of $\MCG(\Sigma)$ with respect to the generating set of all reducible mapping classes.
    \end{corollary}

    \begin{proof}
        The mapping class group acts on $\mathcal{C}(\Sigma)$. In \cite{MM1999} it is shown that $\mathcal{C}(\Sigma)$ is hyperbolic and any pseudo-Anosov mapping class acts hyperbolically. Furthermore, by \cite{Bowditch2008}, this action is acylindrical. In particular, this implies that any hyperbolic element, i.e. any pseudo-Anosov mapping class, is a WPD element. 

        The set of all reducible mapping classes is conjugation-invariant and generates $\MCG(\Sigma)$ \cite{Dehn1987,Lickorish1964}. Furthermore, by definition, a power of each reducible element fixes a curve and thus acts elliptically on $\mathcal{C}(\Sigma)$. Therefore,  \Cref{thm:mainthmwpd} gives the desired result. 
    \end{proof}

%%%%%%%%%%%%%%%%%%%%%%%%%%%%%%%%%%%%%%%%%%%%%%%%%%%%%%%%%%%%%%%%
\subsection{Surface Groups}
%%%%%%%%%%%%%%%%%%%%%%%%%%%%%%%%%%%%%%%%%%%%%%%%%%%%%%%%%%%%%%%%
    Let $\Sigma$ be a closed surface of genus at least $2$. We use point-pushing maps to build an action of $\pi_{1}(\Sigma)$ on the curve graph of $\Sigma$ in order to apply \Cref{thm:mainthmwpd}. Given a point $p \in \Sigma$ we use $\MCG(\Sigma,p)$ to denote the mapping class group of the punctured surface $\Sigma \setminus p$. In this setting we have the Birman Exact Sequence \cite{Birman1969}:
    \begin{align*}
        1 \longrightarrow \pi_{1}(\Sigma,p) \xlongrightarrow{\mathcal{P}} \MCG(\Sigma,p) \xlongrightarrow{\mathcal{F}} \MCG(\Sigma) \longrightarrow 1.
    \end{align*}
    Here $\mathcal{F}$ is the forgetful map obtained by forgetting the puncture $p$ and $\mathcal{P}$ is the \emph{point-pushing} map (see for example \cite[Section 4.2]{FM2012}). The only result about $\mathcal{P}$ we need is the following. %%% Informally, this map can be thought of as placing one's finger at $p$ and pushing along a curve $\gamma \in \pi_{1}(\Sigma,p)$ while dragging the rest of the surface along with it. For more explicit details on this, we direct the reader to \cite[Section 4.2]{FM2012}. We need the following characterization of curves that give rise to pseudo-Anosov mapping classes.

    \begin{theorem}\cite{Kra1981}\label{thm:ptpushpA}
        The point-push map $\mathcal{P}(\gamma) \in \MCG(\Sigma,p)$ is a pseudo-Anosov mapping class if and only if the curve $\gamma \in \pi_{1}(\Sigma,p)$ has nontrivial intersection with every homotopy class of essential simple closed curve on $\Sigma$. 
    \end{theorem}

    We say such a curve \emph{fills} $\Sigma$. 
    
    \begin{corollary}
        Let $\Sigma$ be a closed surface of genus at least two. For all $m \in \N$, $\Z^{m}$ quasi-isometrically embeds into the Cayley graph of $\pi_{1}(\Sigma)$ with respect to the generating set of all non-filling closed curves on $\Sigma$. 
    \end{corollary}

    \begin{proof}
        Let $p$ be a fixed basepoint in $\Sigma$. The point-push map $\mathcal{P}:\pi_{1}(\Sigma,p) \rightarrow \MCG(\Sigma,p)$ gives an action of $\pi_{1}(\Sigma,p)$ on $\mathcal{C}(\Sigma\setminus p)$. Again, this curve graph is hyperbolic and pseudo-Anosov mapping classes act hyperbolically. Also, the acylindricity of the action of $\MCG(\Sigma,p)$ descends to acylindricity of the action of $\pi_{1}(\Sigma,p)$. That is to say, by taking a filling curve in $\pi_{1}(\Sigma,p)$ and applying \Cref{thm:ptpushpA}, we obtain a WPD element for the action of $\pi_{1}(\Sigma,p)$ on $\mathcal{C}(\Sigma\setminus p)$ 

        The set of all non-filling closed curves in $\pi_{1}(\Sigma,p)$ generates and is conjugation-invariant. Additionally, by \Cref{thm:ptpushpA}, the image of any non-filling closed curve in $\MCG(S,p)$ is not pseudo-Anosov. Thus, some power of it must fix some curve in $\mathcal{C}(\Sigma\setminus p)$ and therefore act elliptically. This allows us to apply \Cref{thm:mainthmwpd} to obtain the desired result. 
    \end{proof}

    \begin{remark}
        This strategy of using point-push maps in order to build quasimorphisms on surface groups was also used in \cite{MP} and \cite[Section 4]{bm2019}. %%% In the latter work, Brandenbursky-Marcinkowski needed finer control over the quasimorphisms they build and hence had to make use of the more sophisticated machinery found in \cite{BBF2016}. 
    \end{remark}

%%%%%%%%%%%%%%%%%%%%%%%%%%%%%%%%%%%%%%%%%%%%%%%%%%%%%%%%%%%%%%%%
\subsection{Outer Automorphisms of Free Groups}
%%%%%%%%%%%%%%%%%%%%%%%%%%%%%%%%%%%%%%%%%%%%%%%%%%%%%%%%%%%%%%%%
    We say that an element of $\Out(F_{n})$ is \emph{reducible} if it fixes the conjugacy class of a proper free factor of $F_{n}$ and is \emph{fully irreducible} if no power is reducible. In this setting the appropriate action on a hyperbolic space is provided by the following. 

    \begin{theorem}\cite[Theorem 9.3]{BF2014} \label{thm:OutFnGraph}
        The free factor graph is $\delta$-hyperbolic and every fully irreducible element of $\Out(F_{n})$ is a WPD element. 
    \end{theorem}
    
    \begin{corollary}
        Let $n \geq 2$. For all $m \in \N$, $\Z^{m}$ quasi-isometrically embeds into the Cayley graph of $\Out(F_{n})$ with respect to the generating set of all reducible outer automorphisms. 
    \end{corollary}
    
    \begin{proof}
        Here all of the pieces that we need to apply \Cref{thm:mainthmwpd} are provided by \Cref{thm:OutFnGraph}. The free factor graph is a hyperbolic graph on which $\Out(F_{n})$ acts with a WPD element.  Additionally, every reducible element acts elliptically. Finally we note that set of reducible elements is conjugation-invariant and generates $\Out(F_{n})$ as it contains all of the Nielsen generators \cite{Nielsen1924}. 
    \end{proof}

\subsection{Bestvina-Fujiwara construction}\label{SS:BF}

We first recall some terminology and direct the reader to \cite{bf2002} for more precise statements. An action of a group $G$ on a $\delta$-hyperbolic space $X$ is \emph{non-elementary} if there exist at least two hyperbolic elements whose quasi-axes do not contain rays that stay within bounded distance of each other. For any two hyperbolic elements $g_{1},g_{2} \in G$, we write $g_{1} \sim g_{2}$ if for any arbitrarily long segment in a quasi-axis for $g_{1}$ there is an (orientation-preserving) translation  of this segment so that it lies within a bounded (in terms of the Morse constant) neighborhood of a quasi-axis of $g_{2}$. A hyperbolic element $g$ satisfies the \emph{WPD (weak proper discontinuity)} property if for every $x \in X$ and $\epsilon>0$, there exists $N>0$ such that the set
\begin{align*}
    \{h \in G \vert d(x,h\cdot x) \leq \epsilon \text{ and } d(g^{N}\cdot x,hg^{N} \cdot x) \leq \epsilon\}
\end{align*}
is finite. One interpretation of the WPD condition is that it ensures that the collection of all quasi-axes of conjugates of $g$ is ``discrete'' in $X$. In particular, it implies that $g_{1} \sim g_{2}$ if and only if some positive powers of $g_{1}$ and $g_{2}$ are conjugate \cite[Proposition 6(4)]{bf2002}. 

Next we make use of a construction of quasimorphisms due to Fujiwara \cite{Fujiwara1998} that is modeled on the counting quasimorphisms of Brooks \cite{brooks1981} used in \Cref{SS:Brooks}. We only give a sketch of the construction and direct the reader to \cite[Section 3]{Fujiwara1998} for details. Let $w$ be a finite, oriented path in $X$ and write $|w|$ for the length of $w$. A \emph{copy} of $w$ is any path of the form $g \cdot w$ for $g \in G$. Fujiwara defines a map $h_{w}:G \rightarrow \R$ such that, given a basepoint $x_{0} \in X$, $h_{w}(g)$ coarsely measures the (signed with respect to orientation) number of copies of $w$ seen along a quasi-geodesic connecting $x_{0}$ to $g \cdot x_{0}$. Fujiwara verifies that this map is a quasimorphism in \cite[Proposition 3.10]{Fujiwara1998}. We refer to these quasimorphisms as \emph{Brooks-Fujiwara counting quasimorphisms}. 

The last ingredient from this construction we need is that the homogenizations of these Brooks-Fujiwara counting quasimorphisms vanish on elliptic elements. This follows from \cite[Lemma 3.9]{Fujiwara1998}. Technically, Fujiwara's result is stated in the context of a group acting properly discontinuously on a hyperbolic space. However, throughout \cite[Section 3]{Fujiwara1998}, proper discontinuity is only used to guarantee the existence of a path realizing a certain infimum. With a minor modification, the bounds in \cite[Lemma 3.2]{Fujiwara1998} needed for \cite[Lemma 3.9]{Fujiwara1998} and the following corollary can be obtained without ever appealing to an explicit path. 
\begin{corollary}[Fujiwara]\label{cor:ellipticzero}
    Let $g \in G$ be an elliptic element for the action of $G$ on $X$ and $h_{w}:G \rightarrow \R$ a Brooks-Fujiwara counting quasimorphism. Then $\tilde{h}_{w}(g) = 0$ where $\tilde{h}_{w}$ denotes the homogenization of $h_{w}$. 
\end{corollary}

We are now ready to give a sketch of the proof of \Cref{thm:bf}. We note that the proof is effectively the construction carried out in Sections 2 and 3 of \cite{bf2002} together with \Cref{cor:ellipticzero}. For more details on the construction of infinitely many linearly independent quasimorphisms we direct the reader to the proofs of Theorem 1, Proposition 2, and Proposition 6 in \cite{bf2002}. 

\begin{proof}[Proof of \Cref{thm:mainthmwpd}]
    We begin by running the proof of \cite[Proposition 6(5)]{bf2002}: Let $g_{1}$ be a WPD element and $g_{2}$ another independent hyperbolic element. After passing to powers, the group generated by $g_{1}$ and $g_{2}$ is a free group $F$ so that every nontrivial element acts hyperbolically and the action of $F$ on $X$ is quasi-convex (see \cite[Section 5.2]{Gromov1987} and \cite[Proposition 4.3]{Fujiwara1998}). Now the WPD property of $g_{1}$ ensures that there exists some $f \in F$ so that $g_{1} \not\sim f$. We note that in \cite{bf2002}, the authors assume that every hyperbolic element of $G$ is WPD, however, in the proof of \cite[Proposition 6(5)]{bf2002} one only needs one of the generators of $F$ to be WPD in order to get the desired result. 

    Next, we make use of \cite[Theorem 1 and Proposition 2]{bf2002} to see that we can obtain a sequence of nontrivial, linearly independent quasimorphisms $\{q_{k}:G \rightarrow \R\}_{k=1}^{\infty}$. These quasimorphisms are Brooks-Fujiwara counting quasimorphisms built in the following manner. By \cite[Propostion 2]{bf2002}, we obtain a sequence of hyperbolic elements $f_{1},f_{2},\ldots \in F <G$ such that 
    \begin{enumerate}
        \item $f_{k} \not\sim f_{k}^{-1}$ for $k=1,2,\ldots,$ and 
        \item $f_{k} \not\sim f_{\ell}^{\pm 1}$ for $\ell < k$. 
    \end{enumerate}
    Since $F$ is quasi-convex, we can fix constants $(K,L)$ so that every $f \in F$ has a $(K,L)$-quasi-axis. Properties (1) and (2) ensure that there exists a sequence of powers $a_{k}>0$, depending only on $K,L$, the hyperbolicity constant of $X$, and the translation length of $f_{k}$ so that the quasimorphisms $h_{f_{k}^{a_{k}}}$ are all nontrivial and linearly independent (\cite[Proposition 5]{bf2002}). Here $h_{f_{k}^{a_{k}}}$ denotes the quasimorphism $h_{w_{k}}$ where $w_{k}$ is any choice of geodesic from $x_{0}$ to $f_{k}^{a_{k}} \cdot x_{0}$. Now we set $q_{k} = h_{f_{k}^{a_{k}}}$ for each $k$ and note that, by \Cref{cor:ellipticzero}, $q_{k}(s) = 0$ for all $s \in G$ that acts elliptically on $X$. 
\end{proof}

\bibliography{biblio}
\end{document}